\documentclass{llncs}
\usepackage{amsmath,amssymb,amsfonts,mathrsfs,thm-restate,thmtools}

\usepackage{graphicx}
\usepackage{appendix}
\usepackage{epsf}
\usepackage{epsfig}
\usepackage{epstopdf}
\usepackage{xspace,xcolor}
\usepackage{soul}
\usepackage{latexsym}
\usepackage{tikz}
\usepackage{framed}
\usepackage{todonotes}
\usepackage{lipsum}
\usepackage{setspace}
\usepackage{xcolor}
\usepackage{tabulary}
\usepackage{tabularx}
\usepackage{multirow}
\usepackage{booktabs}

\usepackage{caption}
\usepackage{subcaption}
\usepackage{tkz-tab}
\usetikzlibrary{arrows,positioning}
\usetikzlibrary{shapes,snakes}

\usepackage[bookmarks=false, colorlinks,
linkcolor=red!60!black, citecolor=green!60!black]{hyperref}
\usepackage{nameref}

\usepackage{enumerate}
\usepackage{caption}
\usepackage{subcaption}
\usepackage[compatibility=false]{caption}
\tikzstyle{vertex}=[circle, draw, inner sep=0pt, minimum size=4.5pt]

\title{On Color Critical Graphs of Star Coloring}

\author{ Harshit Kumar Choudhary \and I. Vinod Reddy 
	 \inst{}
}
\institute{Department of Computer Science and Engineering,\\
	Indian Institute of Technology Bhilai, Raipur, India \\
	\email{harshitk@iitbhilai.ac.in,  vinod@iitbhilai.ac.in},
} 

\begin{document}
	\pagestyle{plain}
	
	\maketitle
 \begin{abstract}
 A \emph{star coloring} of a graph $G$  is a proper vertex-coloring such that no
path on four vertices is $2$-colored. The minimum number of colors required to obtain a star coloring of a graph $G$ is called star chromatic number and it is denoted by $\chi_s(G)$. A graph $G$ is called $k$-critical if $\chi_s(G)=k$ and $\chi_s(G -e) < \chi_s(G)$ for every edge $e \in E(G)$. In this paper,  we give a  characterization of 3-critical, $(n-1)$-critical and $(n-2)$-critical graphs with respect to star coloring, where $n$ denotes the number of vertices of $G$.  We also give upper and lower bounds on the minimum number of edges in $(n-1)$-critical and $(n-2)$-critical graphs.

\end{abstract}


\section{Introduction}
All graphs considered in this paper are finite, undirected, connected and simple (without loops and multiple edges). For a graph $G=(V,E)$, we use $V(G)$ and $E(G)$ to denote the vertex set and edge set of $G$ respectively. A proper $k$-coloring of a graph $G$ is a mapping $f:V(G) \rightarrow \{1,2,\ldots, k\}$ such that $f(u) \neq f(v)$  for every edge $e=uv$ of $G$. If $G$ has a proper $k$-coloring, then $G$ is
said to be $k$-colorable.  The smallest integer $k$ such that $G$ is $k$-colorable is called the chromatic number of $G$, denoted by $\chi(G)$. A $k$-star coloring of a graph $G$ is a
proper $k$-coloring of $G$ such that every path on four vertices uses at least three distinct colors. The smallest
integer $k$ for which $G$ admits a $k$-star coloring is called the star chromatic number of $G$, denoted by $\chi_s(G)$. The name star coloring is given due to the fact that the subgraph induced by the union of any
two color classes (a subset of vertices assigned the same color) is a disjoint union
of stars.

Star coloring of graphs was introduced by Gr{\"u}nbaum in~\cite{grunbaum1973acyclic}. Star coloring of planar graphs is well-studied. Any planar graph can be star colored with $80$ colors~\cite{grunbaum1973acyclic}. Later in 2003, Ne{\v{s}}et{\v{r}}il and Ossona de Mendez~\cite{nevsetvril2003colorings} improved the result by showing that any planar graph has a star coloring with $30$ colors. They also showed that any planar bipartite graph can be star colored with $18$ colors. Further, Albertson et al.~\cite{albertson2004coloring} showed that any planar graph has a star coloring with $20$ colors. 
Guillaume Fertin et al.~\cite{fertin2004star} studied the star coloring of various graph classes such as trees, cycles, complete bipartite graphs, outerplanar graphs and  $2$-dimensional grids. It is {\sf NP}-complete to determine whether $\chi_s(G) \leq 3$, even on planar bipartite graphs~\cite{albertson2004coloring}. Star coloring problem is polynomial-time solvable on cographs~\cite{lyons2011acyclic}, line graphs of tress~\cite{omoomi2018polynomial}. For all $k \geq 3$, the $k$-star coloring is {\sf NP}-complete on bipartite graphs~\cite{coleman1983estimation}.  $3$-star coloring is {\sf NP}-complete on planar bipartite graphs~\cite{albertson2004coloring} and line graphs of subcubic graphs~\cite{lei2018star}. 

A graph $G$ is called $k$-critical if  $\chi_s(G)=k$ and $\chi_s(G-e) < \chi_s(G)$ for every edge $e \in E(G)$.
Color critical graphs with respect to proper coloring have been well studied~\cite{dirac1957theorem,kostochka1999excess,krivelevich1997minimal,gao2023tight}. However,
to the best of our knowledge, the color-critical graphs with respect to star coloring have not been explored. In this paper, we initiate the study of color-critical graphs under star coloring.

\paragraph{Our contributions.}
In this paper, we prove the following results. 
\begin{theorem}
 A graph $G$ is $3$-critical if and only if  $G$ is either $K_3$ or $P_4$. 
\end{theorem}

\begin{theorem}
  Let $G$ be a  non-complete graph on $n \geq 5$ vertices.  $G$ is $(n-1)$-critical if and only if $G$ is $(I_3, 2K_2)$-free and $G-e$ contains either an $I_3$ or $2K_2$ for every edge $e$ of $G$. 
\end{theorem}

\begin{theorem}
 
Let $G$ be a graph on $n \geq 5$ vertices and $G$ contains either an $I_3$ or $2K_2$. $G$ is $(n-2)$-critical  if and only if
$G$ is $(I_4, 2K_2+K_1, P_3+P_2)$-free and $G-e$ contains one of  $I_4$ or $2K_2+K_1$ or $P_3+P_2$ as an induced subgraph for every edge $e$ of $G$. 
\end{theorem}

\begin{theorem}
  Let $G$ be a graph on $n \geq 5$ vertices and $m$ edges. If $G$ is $(n-1)$-critical then
 $$ \frac{n^2-n-n \sqrt{n}}{2} < m\leq(n-1)(n-2)/2$$

\end{theorem}

\begin{theorem}
  Let $G$ be a graph on $n \geq 5$ vertices and $m$ edges. If $G$ is $(n-2)$-critical then
 $$ n(n-3)/6 \leq m\leq n(n-3)/2$$

\end{theorem}

\section{Notation and terminology}
In this section, we review some basic definitions and notations used in this paper. 
We use $[k]$ to denote the set $\{1,2, \ldots,k\}$. If $f:V(G) \rightarrow[k]$ is a $k$-star coloring of $G$, then we use $f^{-1}(i)$ to denote the subset of vertices of $G$ which are assigned the color $i$. 
A graph $H$ is called a \emph{subgraph}
of a graph $G$, if $V(H) \subseteq V(G)$ and $E(H) \subseteq E(G)$. 
A subgraph $H$ of $G$ is called \emph{induced subgraph} of $G$, if for any two vertices $u,v$  in $H$, $u$ and $v$ are adjacent in $H$ if they are adjacent in $G$. 
Given a graph $G$ and a vertex $v \in V(G)$, we use $G -v$ to denote the graph obtained by removing $v$ and all edges incident on $v$ from $G$. 
For a vertex subset  $X \subseteq V(G)$, $G - X$ denotes the graph obtained by deleting vertices of $X$ and all edges incident on $X$ from $G$. Similarly, if $e \in E(G)$,  we use $G -e$ to denote the graph obtained by removing the edge $e$ from $G$. 
The open neighborhood of a vertex $v$, denoted
$N(v)$, is the set
of vertices adjacent to $v$ and the set $N[v] = N(v) \cup \{v\}$ denote the closed neighborhood of $v$. 
The size of the set $N(v)$ is called the \emph {degree} of $v$ in G, denoted as $deg(v)$. The maximum degree of a graph $G$, denoted by $\Delta(G)$.

A set $I \subseteq V(G)$ of pairwise nonadjacent vertices is called an \emph{independent set}. We use $I_k$ to denote the independent set of size $k$.  A subset $C$ of $V(G)$ is called a \emph{clique} if any two vertices in $C$ are adjacent.
A \emph{star graph} is a tree with at most one vertex with degree greater than one. We use $P_{n}$ and $K_n$ respectively to denote a path and a complete graph  on $n$ vertices. The disjoint union of two graphs $G$ and $H$ is denoted as $G +H$. A graph $G$ is $H$-free if it does not contain $H$ as an induced subgraph. A graph $G$ is $2K_2$-free if it does not contain two independent edges as an induced subgraph. A graph $G$ is $(P_3+P_2)$-free if it does not contain any induced subgraph which is a disjoint union of $P_3$ and $P_2$. For more details on standard graph-theoretic
notation and terminology, we refer the reader to the textbook~\cite{diestel2005graph}. The graphs which appear commonly in this paper are shown in Fig~\ref{fig-1}.
\section{$3$-critical graphs}
It is easy to see that $K_2$ is the only $2$-critical connected graph and $K_n$ is the only $n$-critical connected graph. 
In this section, we focus on  graphs that are $3$-critical.

\begin{figure}[t]
	\begin{subfigure}[t]{.20\textwidth}
	\caption{$I_3$}
		$$
			\includegraphics[trim=9.8cm 14cm 9cm 12cm, clip=true, scale=1]{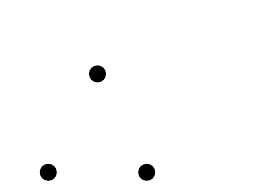}
		$$
		
	\end{subfigure}
	\begin{subfigure}[t]{.19\textwidth}
	\caption{$2K_2$}
		$$
			\includegraphics[trim=9.5cm 14cm 9cm 12cm, clip=true, scale=1]{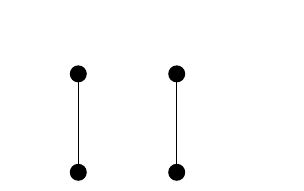}
		$$
		
		\end{subfigure}
		\begin{subfigure}[t]{0.19\textwidth}
		\caption{$I_4$}
			$$
				\includegraphics[trim=9.4cm 14cm 9cm 12cm, clip=true, scale=1]{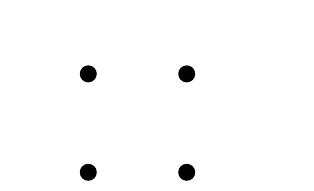}
			$$
		
		\end{subfigure}
		\begin{subfigure}[t]{0.19\textwidth}
		\caption{$2K_2+K_1$}
			$$
				\includegraphics[trim=9.4cm 14cm 9cm 12cm, clip=true, scale=1]{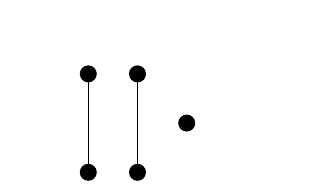}
			$$
		
		\end{subfigure}
			\begin{subfigure}[t]{0.19\textwidth}
		\caption{$P_3+P_2$}
			$$
				\includegraphics[trim=9.4cm 13.5cm 9cm 12cm, clip=true, scale=1]{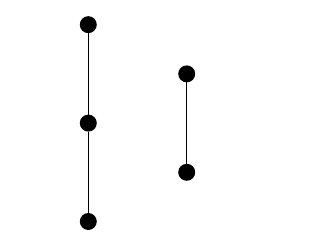}
			$$
		
		\end{subfigure}
	\caption{(a) $I_3$: Independent set of size three; (b) $2K_2$: Disjoint union of two complete graphs on two vertices; (c) $I_4$: Independent set of size four; (d) $2K_2+K_1$: Disjoint union of two complete graphs on two vertices and one complete graph on one vertex; and (e) $P_3+P_2$: Disjoint union of a path on three vertices and a path on two vertices.}
\label{fig-1}
\end{figure}

It is known that a graph is  $3$-critical with respect to proper coloring if and only if  it is an odd cycle. Here, we show that the class of $3$-critical graphs with respect to star coloring are only $K_3$ and $P_4$.
$K_3$ is the only $3$-critical graph having at most three vertices. Next, we show that $P_4$ is the only $3$-critical graph having at least four vertices. 
\begin{lemma}\label{lem-p4}
 Let $G$ be a connected graph having at least four vertices. $G$ is a star graph if and only if $G$ does not contain $P_4$ as a subgraph.
\end{lemma}
\begin{proof}
 If $G$ is a star graph then at most one vertex has degree greater than one,  hence $G$ does not contain any $P_4$ as a subgraph. For the other direction, let $G$ be a graph without any $P_4$ as a subgraph. Suppose $G$ is not a star graph then $G$ contains two vertices $u$ and $v$ having degree at least two. If $N(u) \cap N(v) \neq \emptyset$ then $G$ contains $P_4$ as subgraph. If $N(u) \cap N(v) = \emptyset$ then as $G$ is connected there is a path between $u$ and $v$ of length at least three, hence $G$ contains $P_4$ as a subgraph, which is a contradiction. Therefore $G$ is a star graph.
\end{proof}

\begin{theorem}\label{thm-3-critical}
Let $G$ be a connected graph having at least four vertices. $G$ is $3$-critical if and only if $G=P_4$.
\end{theorem}
\begin{proof}
 Clearly if $G=P_4$ then $\chi_s(G)=3$ and $\chi_s(G-e)\leq 2$ for every $e \in E(G)$. For the other direction, let $G$ be a $3$-critical graph, then $G$ contains $P_4$ as a subgraph.  Otherwise  from Lemma~\ref{lem-p4}, $G$ is a star graph and $\chi_s(G)=2$, which is a contradiction.  Let $v_1,v_2,v_3,v_4$ be a $P_4$ in $G$. If $G$ is not $P_4$, then it is $P_4$ with some additional vertices and edges. If $G$ has a vertex $u$ other than $v_1$, $v_2$, $v_3$, and $v_4$, then after deleting any edge $e$ incident on $u$ from $G$, it still contains $P_4$ as a subgraph, that is $\chi_s(G-e)=3$, which is a contradiction. Therefore, $G$ contains exactly four vertices. If $G$ contains an edge $e$ other than $v_1v_2$, $v_2v_3$, and $v_3v_4$, then again after deleting the edge $e$ from $G$, it still contains $P_4$ as a subgraph,  which is a contradiction. Therefore, $G$ contains exactly four vertices $\{v_1,v_2,v_3,v_4\}$ and three edges $\{v_1v_2, v_2v_3, v_3v_4\}$. Hence, the only four vertex $3$-critical graph is $P_4$.
\end{proof}

\section{$(n-1)$-critical graphs}\label{sec:n-1}
First, we show that there does not exist any $(n-1)$-critical graph under proper coloring. Later, we show the existence of $(n-1)$-critical graphs with respect to star coloring. 

\begin{lemma}
 Let $G$ be a graph with $\chi(G)=n-1$ then $G$ is not $(n-1)$-critical.
\end{lemma}
\begin{proof}
 Let $f:V(G) \rightarrow [n-1]$ be a $(n-1)$ coloring of $G$. Then, there exists two vertices $x$ and $y$ in $G$ such that $f(x)=f(y)$. Observe that every vertex of $V(G) - \{x,y\}$ either adjacent to $x$ or $y$, otherwise, $G$ has an independent set of size three, hence $\chi(G) \leq n-2$, which is a contradiction. 
 
 \emph{Cliam}. Either $N(x)=V(G) - \{x, y\}$ or $N(y)=V(G) - \{x,y\}$
  
  Suppose not, that is there exists $u \in N(x) \setminus N(y)$ and $v \in N(y) \setminus N(x)$. By coloring $y$ with the color of $u$ and $x$ with the color of $v$, we get $\chi(G) \leq n-2$, which is a contradiction. Therefore, at least one of $x$ or $y$ is adjacent to every vertex of $V(G)-\{x,y\}$. Without loss of generality, assume that $N(x)=V(G) - \{x,y\}$.  As $\chi(G)=n-1$, there exists an edge between every pair of color classes, that is $K_{n-1}$ is a subgraph of $G$. Then the graph obtained by deleting any edge incident on $y$, contains $K_{n-1}$ as a subgraph of $G$. Hence, $G$ is not $(n-1)$-critical. \qed
\end{proof}

Next, we show that there exists $(n-1)$-critical graphs with respect to star coloring.   
We give a characterization of $(n-1)$-critical graphs under star coloring. 

\begin{lemma}\label{lem-free}
 Let $G$ be a non-complete graph on $n \geq 5$ vertices. $\chi_s(G)=n-1$ if and only if $G$ is $(I_3, 2K_2)$-free. 
\end{lemma}
\begin{proof}
 Let $G$ be a graph such that $\chi_s(G)=n-1$. Suppose $G$ has an $I_3$ then by coloring all vertices in $I_3$ with one color and all other vertices with distinct colors we get a $(n-2)$-star coloring of $G$, hence $\chi_s(G) \leq n-2$, which is a contradiction. Suppose $G$ has a $2K_2$ then by coloring all vertices in $2K_2$ with two colors and all other vertices with distinct colors we get a $(n-2)$-star coloring of $G$, hence $\chi_s(G) \leq n-2$, which is again a contradiction. Therefore $G$ is $(I_3, 2K_2)$-free. 
 
 For the reverse direction, assume that $G$ is $(I_3, 2K_2)$-free.  As $G$ is non-complete, clearly $\chi_s(G)\leq n-1$. Suppose $\chi_s(G)= \ell \leq n-2$. Let $f$ be a $\ell$-star coloring of $G$. If there exists an $i \in [\ell]$ such that $|f^{-1}(i)|\geq 3$, then $G$ has an $I_3$. Hence, we can assume that $|f^{-1}(i)|\leq 2$ for all $i \in [\ell]$. As $\ell \leq n-2$, there exists $i, j \in [\ell]$, such that $|f^{-1}(i)|= 2$ and $|f^{-1}(j)|= 2$. Let $f^{-1}(i)=\{u_1, u_2\}$ and $f^{-1}(j)=\{v_1, v_2\}$. Observe that $u_1$ (resp. $u_2)$ is adjacent to at least one of $v_1$ or $v_2$, otherwise $\{u_1,v_1,v_2\}$ (resp. $\{u_2,v_1,v_2\}$)  forms an $I_3$. Similarly, $v_1$ (resp. $v_2$) is adjacent to at least one of $u_1$ or $u_2$.  As $G$ is $2K_2$-free, there are at least three edges between $\{u_1,u_2\}$ and  $\{v_1,v_2\}$. Then we get bi-colored $P_4$, in $G$ with colors $i$ and $j$, which is a contradiction to the fact that $f$ is a star coloring of $G$. Hence $\chi_s(G) =n-1$. 
\end{proof}

\begin{corollary}
  Let $G$ be a  non-complete graph on $n \geq 5$ vertices.  $G$ is $(n-1)$-critical if and only if $G$ is $(I_3, 2K_2)$-free and $G-e$ contains either an $I_3$ or $2K_2$ for every edge $e$ of $G$. 
\end{corollary}

\begin{figure}[t]
\centering

\includegraphics[trim=6cm 12.5cm 6cm 11cm, clip=true, scale=1]{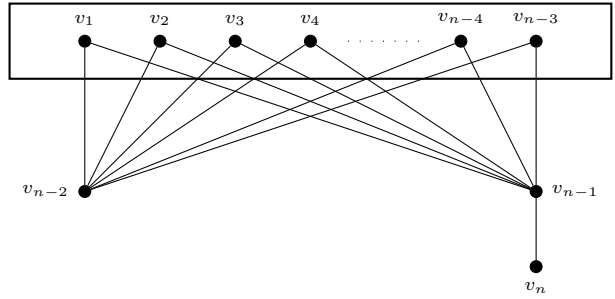}
	
\caption{Horn graph $H_n$. The vertices $\{v_1,v_2,\ldots, v_{n-3}\}$ forms a clique of size $n-3$.}
\end{figure}
Now, we explicitly describe a class of $(n-1)$-critical graphs. We call them as horn graphs. 
\subsection{Horn graphs: a class of $(n-1)$-critical graphs}
For $n \geq 5$, a horn graph $H_n$ on $n$-vertices is defined as follows.
$$V(H_n)=\{v_1,v_2, \ldots, v_n\}$$
$$E(H_n)=\{v_iv_j~|~ 1 \leq i \leq n-3, i+1 \leq j \leq n-1\} \cup \{v_{n-1}v_n\}$$

\begin{lemma}
Let $H_n$ be a horn graph then $\chi_s(H_n)=n-1$.
\end{lemma}
\begin{proof}
 First observe that $U=\{v_1, \ldots, v_{n-2}\}$ induces a clique of size $n-2$ in $H_n$, that is $\chi_s(H_n) \geq n-2$. 
 Suppose $\chi_s(H_n) = n-2$, let $f$ be a $(n-2)$- star coloring of $H_n$. With out loss of generality assume that $f(v_i)=i$ for $i \in [n-2]$. As $v_{n-1}$ is adjacent to every $v_i$, for $i \in [n-3]$, we have $f(v_{n-1})=n-2$. If $f(v_n)=j$ for some $j \in [n-3]$, then $H_n$ contains a bi-colored $P_4$ with colors $(n-2)$ and $j$. This is a contradiction to the fact that $f$ is star coloring. Therefore $\chi_s(H_n) \geq n-1$.
 
 Next we give a $(n-1)$-star coloring of $H_n$. Let $g: V(H_n) \rightarrow [n-1]$ defined as follows. 
$g(v_i)=i$ for $i \in [n-1]$ and $g(v_n)=1$. As $|g^{-1}(1)|=2$ and $|g^{-1}(i)|=1$ for $i \in \{2, \ldots n-1\}$, $g$ is a $(n-1)$-star coloring of $H_n$. Hence $\chi_s(H_n) = n-1$.
\end{proof}
\
\begin{lemma}
Let $H_n$ be an horn graph then $H_n$ is $(n-1)$-critical.
\end{lemma}
\begin{proof}
 We show that for every edge $e$ of $H_n$, $\chi_s(H_n-e)<n-1$.
 If $e=v_nv_{n-1}$, then $H_n-e$ has an independent set $\{v_{n-2}, v_{n-1}, v_n\}$ of size $3$, hence from Lemma~\ref{lem-free}, $\chi_s(G) \leq n-2<n-1$. If $e=v_{n-1}v_j$ for some $j \in [n-2]$, then the edges $v_{n-1}v_n$ and $v_jv_{n-2}$ forms a $2K_2$, hence from Lemma~\ref{lem-free}, $\chi_s(G) \leq n-2 <n-1$.
 If $e=v_iv_j$, for some $i, j \in [n-2], i \neq j$, then $H_n-e$ has an independent set $\{v_{i}, v_{j}, v_n\}$ of size $3$, hence from Lemma~\ref{lem-free}, $\chi_s(G) \leq n-2<n-1$. Hence $H_n$ is $(n-1)$-critical.
\qed
\end{proof}

\subsection{Minimum and maximum number of edges in $(n-1)$-critical graphs}
The problem of finding the minimum number of edges in a $k$-critical graph on $n$ vertices with respect to proper coloring has been  studied in~\cite{dirac1957theorem,krivelevich1997minimal}. 
In this subsection, we give upper and lower bounds on the number of edges in a $(n-1)$-critical graph with respect to star coloring. 

\begin{lemma}\label{lem-not-critical}
 Let $G$ be a graph on $n \geq 5$ vertices such that $\chi_s(G)=\Delta(G)=n-1$, then $G$ is not $(n-1)$-critical.
\end{lemma}
\begin{proof}
Let $f$ be a $(n-1)$-star coloring of $G$. As $\chi_s(G)=n-1$, there exists two vertices $v_i, v_j \in V(G)$ such that $f(v_i)=f(v_j)$. Let $w \in V(G)$ be a vertex having the degree $\Delta(G)$, that is $w$ is adjacent to every other vertex of $G$. Let $v_k$ be a neighbor of $w$, where $ k \neq i, j$.
Let $H=G - wv_k$

\emph{Claim.} $\chi_s(H)=n-1$.

It is enough to show that $H$ is $(I_3, 2K_2)$-free (because of Lemma~\ref{lem-free}). If $H$ has an $I_3$ then $w \notin I_3$, as degree of $w$ in $H$ is $n-2$.  Then the same $I_3$ is also present in $G$, which is a contradiction. If $H$ has a $2K_2$ then $w \notin V(2K_2)$, that is, none of the edges of $2K_2$ are incident on $w$, which implies there is a $2K_2$ in $G$ which is a contradiction. Therefore $H$ is $(I_3,2K_2)$-free and from Lemma~\ref{lem-free}, we get $\chi_s(H)=n-1$.
\end{proof}

\begin{lemma}\label{lem-max-degree}
 Let $G$ be a graph on $n \geq 5$ vertices and $m$ edges. If $m > (n-1)(n-2)/2$ then $\Delta(G) \geq n-2$. 
\end{lemma}
\begin{proof}
 Using the Handshaking lemma, we get 
 $n \Delta(G) \geq 2m>(n-1)(n-2)$. That is $\Delta(G) >(n-1)(n-2)/n=(n-3)+2/n$. Hence, $\Delta(G) \geq n-2$.
\end{proof}

\begin{theorem}
 Let $G$ be a graph on $n \geq 5$ vertices, $m$ edges with $\chi_s(G)=n-1$. If $m > (n-1)(n-2)/2$  then $G$ is not $(n-1)$-critical.
\end{theorem}
\begin{proof}
 We prove it by using induction on number vertices.

 \emph{Base case:} Let $G$ be a connected graph with $n=5$ vertices, and $m>6$, then we can easily see that there is no $4$-critical graph on $5$ vertices having at least $7$ edges. Hence the base case holds.
 
 Assume that the statement is true for all graphs having at most $k-1$ vertices. 
 Let $G$ be a connected graph on $k$ vertices, $m > (k-1)(k-2)/2$  and $\chi_s(G)=k-1$. We need to show that $G$ is not $(k-1)$-critical.
 
 From Lemma~\ref{lem-max-degree}, we know that $\Delta(G) \geq k-2$. If $\Delta(G)=k-1$ them from Lemma~\ref{lem-not-critical}, we know that $G$ is not $(k-1)$-critical. Therefore we assume that $\Delta(G)=k-2$.

 Let $v$ be a vertex such that $deg(v)=k-2$.  Let $H=G-v$ then $|V(H)|=k-1$.
 Notice that $H$ is not a complete graph otherwise, we get $\Delta(G)=k-1$, hence $\chi_s(H)=k-2$. 
 
 $$|E(H)|>(k-1)(k-2)/2 -(k-2)>(k-2)(k-3)/2$$
 
 By inductive hypothesis, $H$ is not $(k-2)$-critical. That is there exists an edge $e \in E(H)$ such that $\chi_s(H)=\chi_s(H-e)=k-2$. From Lemma~\ref{lem-free}, the graphs $H$ and $H-e$ are $(I_3, 2K_2)$-free. 
 
 \emph{Claim:} $\chi_s(G-e)=k-1$
 
 It is enough  to show that $G$ is $(I_3,2K_2)$-free. Suppose $G$ has an $I_3$. Then $v \notin I_3$, as any independent set containing $v$ has size two.
 Which implies there is an $I_3$ in $H$, which is a contradiction.
 
 Similarly, If $G$ has a $2K_2$,  then $v \notin V(2K_2)$ as $deg(v) =k-2$. Which implies there is a $2K_2$ in $H$, which is a contradiction. Hence $G-e$ is $(I_3, 2K_2)$-free. Therefore, using Lemma~\ref{lem-free} we conclude that $G$ is not $(n-1)$-critical. 
\end{proof}

\begin{theorem}
 Let $G$ be a graph on $n \geq 5$ vertices and $m$ edges. If $\chi_s(G)=n-1$ then  $m > \frac{n^2-n-n \sqrt{n}}{2}$.  
\end{theorem}
\begin{proof}
Let $G$ be a graph with $\chi_s(G)=n-1$. We know from Lemma~\ref{lem-free}, $G$ is $(I_3,2K_2)$-free. Then the complement $\overline {G}$ of $G$ is $(C_3,C_4)$-free. In~\cite{van2001course} it is showed that the upper bound on the maximum number of edges in a $(C_3,C_4)$-free graph  is $\frac{n \sqrt{n-1}}{2}$. Hence, the minimum number of edges in a graph having star chromatic number $(n-1)$ is at least $ {n \choose 2} -  \frac{n \sqrt{n-1}}{2} > \frac{n^2-n-n \sqrt{n}}{2}$. 

%
%
\end{proof}

\begin{corollary}
 Let $G$ be a graph on $n \geq 5$ vertices and $m$ edges. If $G$ is $(n-1)$-critical then
 
 $$ \frac{n^2-n-n \sqrt{n}}{2} < m\leq(n-1)(n-2)/2$$
\end{corollary}

\begin{remark}
The upper bound given in the above corollary is tight as horn graph $H_n$ is $(n-1)$-critical and contains $(n-1)(n-2)/2$ edges. 
\end{remark}

\section{$(n-2)$-critical graphs}\label{sec:n-2}

In this section, we give a characterization of $(n-2)$-critical graphs with respect to  star coloring.

\begin{lemma}\label{lem-free-n-2}
 Let $G$ be a  graph on $n \geq 5$ vertices and $G$ contains either an $I_3$ or $2K_2$. $\chi_s(G)=n-2$ if and only if $G$ is $(I_4, 2K_2+K_1, P_3+P_2)$-free. 
\end{lemma}
\begin{proof}
  Let $G$ be a graph such that $\chi_s(G)=n-2$. Suppose $G$ has a $I_4$ then by coloring all the vertices in $I_4$ with one color and all other vertices with distinct colors we get a $(n-3)$-star coloring of $G$, hence $\chi_s(G) \leq n-3$, which is a contradiction, as $\chi_s(G)=n-2$. Suppose $G$ has a $2K_2+K_1$ then by coloring all vertices in $2K_2+K_1$ with two colors and all other vertices with distinct colors we get a $(n-3)$-star coloring of $G$, hence $\chi_s(G) \leq n-3$, which is again a contradiction. Suppose $G$ has a $P_3+P_2$ then by coloring all vertices of $P_3+P_2$ with two colors and all other vertices with distinct colors we get $(n-3)$-star coloring of $G$, hence $\chi_s(G) \leq n-3$, which is again a contradiction.
  Therefore $G$ is $(I_4, 2K_2+K_1, P_3+P_2)$-free. 
 
 For reverse direction, assume that $G$ is $(I_4, 2K_2+K_1, P_3+P_2)$-free. As $G$ contains either an $I_3$ or $2K_2$, we have $\chi_s(G) \leq n-2$. Suppose $\chi_s(G)= \ell $. Let $f$ be a $\ell$-star coloring of $G$. As $G$ is $I_4$-free, we have $|f^{-1}(i)|\leq 3$ for all $i \in [\ell]$. 
 
%
%
%
 
 \paragraph{Case~1.} There exists $i, j \in [\ell]$, such that $|f^{-1}(i)|= 3$ and $|f^{-1}(j)|= 2$.

  Let $f^{-1}(i)=\{u_1, u_2, u_3\}$ and  $f^{-1}(j)=\{v_1, v_2\}$. Let $H$ be the subgraph of $G$ induced by the vertices  $f^{-1}(i) \cup f^{-1}(j)$. Clearly, $H$ is a disjoint union of stars. If $H$ contains more than one component then it will contain one of $I_4$, $2K_2+K_1$, $P_3+P_2$ as a subgraph. As $G$ is $(I_4, 2K_2+K_1, P_3+P_2)$-free, hence $H$ must be connected. 
  If $H$ is connected, then, we get bi-colored $P_4$, in $G$ with colors $i$ and $j$, which is a contradiction to the fact that $f$ is a star coloring of $G$.

 \paragraph{Case~2.} There exists $i, j, k \in [\ell]$, such that $|f^{-1}(i)|= 2$, $|f^{-1}(j)|= 2$ and $|f^{-1}(k)|= 2$.

  Let $f^{-1}(i)=\{u_1, u_2\}$,  $f^{-1}(j)=\{v_1, v_2\}$ and  $f^{-1}(k)=\{w_1, w_2\}$. Let $H$ be the subgraph of $G$ induced by the vertices  $f^{-1}(i) \cup f^{-1}(j)\cup f^{-1}(k)$.
 Between any two color classes we can have at most two edges, otherwise, we get a bicolored $P_4$. Hence, $H$ can have at most six edges. Also, as $H$ is $I_4$-free, it contains at least three edges. That is, the number of edges in $H$ is at least three and at most six. In each of these cases, $H$ contains either one of $I_4$, $2K_2+K_1$, $P_3+P_2$ as an induced subgraph or a bicolored $P_4$. In each of the cases we get a contradiction to  either $f$ is a star coloring of $G$ or $G$ is $(I_4$, $2K_2+K_1$, $P_3+P_2)$-free.

 Combing the above two cases, one of the following must hold.

 \begin{enumerate}
  \item If there exists $i \in [\ell]$ such that $|f^{-1}(i)|=3$ then $|f^{-1}(j)|=1$ for every $j\in [\ell]\setminus\{i\}$. In this case clearly $f$ uses at least $n-2$ colors.
  
  \item If there exists $i,j \in [\ell]$ such that $|f^{-1}(i)|= 2$ and $|f^{-1}(j)|= 2$ then $|f^{-1}(k)|=1$ for every $k\in [\ell]\setminus\{i,j\}$ . In this case also  $f$ uses at least $n-2$ colors.
 \end{enumerate}
 
 We conclude that any star coloring $f$ of $G$ uses at least $n-2$ colors. Hence we get $\chi_s(G)=n-2$.   \qed

\end{proof}
\begin{corollary}
 Let $G$ be a  graph on $n \geq 5$ vertices and $G$ contains either an $I_3$ or $2K_2$. $\chi_s(G)=n-2$ if and only if
$G$ is $(I_4, 2K_2+K_1, P_3+P_2)$-free and $G-e$ contains one of  $I_4$ or $2K_2+K_1$ or $P_3+P_2$ as an induced subgraph for every edge $e$ of $G$. 
\end{corollary}

 In Section~\ref{sec:n-1}, we have showed that $(n-1)$-critcal graphs under proper coloring does not exist, however there exists graphs which are $(n-2)$-critical under proper coloring. For example, $C_5$ is $(n-2)$-critical. In general, given any $n \geq 5$, we can construct a $(n-2)$-critical graph $G_n$ as follows. Intially set $G_5=C_5$ and for $i \geq 6$, $G_{i}$ is obtained from $G_{i-1}$ by adding a new vertex $v_i$ adjacent to all the vertices of $G_{i-1}$. It is easy to see that $G_n$ is $(n-2)$-critical.

 Next, we provide a class of $(n-2)$-critical with respect to star coloring.  We call them as double-horn graphs.

\subsection{Double horn graphs: a class of $(n-2)$-critical graphs} 
 
 For $n \geq 5$, a double horn graph $D_n$ on $n$-vertices is defined as follows.
$V(D_n)=\{v_1,v_2, \ldots, v_n\}$
and 
$E(D_n)=\{v_iv_j~|~ i, j \in [n-4], i \neq j\} \cup \{v_{n-3}v_i~|~ i \in [n-5]\} \cup \{v_{n-2}v_i~|~ i \in \{2,3, \ldots, n-4\}\}\cup\{v_{n-2}v_{n-3},v_{n-1}v_{n-3}, v_{n}v_{n-2}\}$


\begin{lemma}
Let $D_n$ be a double-horn graph then $\chi_s(D_n)=n-2$.
\end{lemma}
\begin{proof}
 First observe that $U=\{v_1, \ldots, v_{n-4}\}$ induces a clique of size $n-4$ in $D_n$, that is $\chi_s(D_n) \geq n-4$. 
 Suppose $\chi_s(D_n) =  n-3$, let $f$ be a $(n-3)$- star coloring of $D_n$. With out loss of generality assume that $f(v_i)=i$ for $i \in [n-4]$.

\begin{figure}[t]
\centering

\includegraphics[trim=6cm 12.5cm 6cm 11cm, clip=true, scale=1]{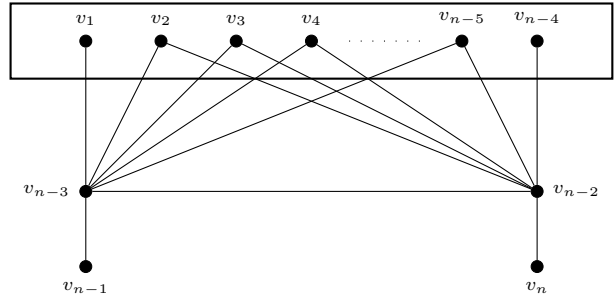}
	
\caption{Double-horn graph $D_n$. The vertices $\{v_1,v_2,\ldots, v_{n-4}\}$ forms a clique of size $n-4$.}
\end{figure}

 \paragraph{Case~1.} $f(v_{n-3})=n-4$.
 
 As $v_{n-2}$ is adjacent to $v_2, v_3 \ldots, v_{n-4}  $, so $f(v_{n-2}) \neq i$, for any $ i \in \{2, \ldots, n-4\}$. If $f(v_{n-2})=1$, then we get bicolored $P_4$ with colors $1$ and $n-4$. Hence $f(v_{n-2})=n-3$. If $f(v_{n-1})= p$, where $p \in [n-3], p \neq n-4$, then we get a bicolored $P_4$ with colors $p$ and $n-4$, which is a contradiction to the fact that $f$ is a $(n-3)$-star coloring of $D_n$.

  \paragraph{Case~2.} $f(v_{n-3})=n-3$.

 As $v_{n-2}$ is adjacent to $v_2, v_3 \ldots, v_{n-4}, v_{n-3}  $, so $f(v_{n-2}) \neq i$, for any $ i \in \{2, \ldots, n-4, n-3\}$. Hence $f(v_{n-2})=1$.  If $f(v_{n})= p$, where $p \in [n-3], p \neq 1$, then we get a bicolored with colors $p$ and $1$, which is a contradiction to the fact that $f$ is a $(n-3)$-star coloring of $D_n$.
 
From the above two cases, we conclude that  $\chi_s(D_n) \geq n-2$.
 
 Next, we give a $(n-2)$-star coloring of $D_n$. Let $g: V(D_n) \rightarrow [n-2]$ defined as follows. 
$g(v_i)=i$ for $i \in [n-4]$, $g(v_{n-3})=n-4$, $g(v_{n-1})=g(v_{n})=n-3$, and $g(v_{n-2})=n-2$. 
Observe that $|g^{-1}(n-4)|=|g^{-1}(n-3)|=2$ and $|g^{-1}(i)|=1$ for $i \in [n]\setminus\{ n-3, n-4\}$. Clearly $g$ is a $(n-2)$-star coloring of $D_n$, as there is no $P_4$ containing the vertices that are colored with $n-3$ and $n-4$.  Hence $\chi_s(D_n) = n-2$.
\end{proof}

\begin{lemma}
Let $D_n$ be an double-horn graph then $D_n$ is $(n-2)$-critical.
\end{lemma}
\begin{proof}
 We show that for every edge $e$ of $D_n$, the graph $D_n-e$ contains one of $I_4$, $2K_2+K_1$ or $P_3+P_2$ as induced subgraph. Then the result follows from Lemma~\ref{lem-free-n-2}.

\paragraph{Case~1.} If $e=v_iv_{n-3}$, $i \in [n-5]$

In this case $v_{n-3}v_{n-1}$, $v_iv_{n-4}$ and $v_n$ forms a $2K_2+K_1$.

\paragraph{Case~2.} If $e=v_iv_{j}$, for $i, j \in [n-4]$, $i \neq j$.

In this case $\{v_i, v_j, v_{n-1}, v_{n}\}$ forms an $I_4$.

\paragraph{Case~3.} If $e=v_{n-2}v_{n-3}$.

In this case the vertices $\{v_{1}, v_{n-3}, v_{n-1}\}$ and $\{v_n,v_{n-2}\}$ forms $P_3+P_2$.

\paragraph{Case~4.} If $e=v_{n-1}v_{n-3}$.

In this case $\{v_{n-1}, v_{n-3}, v_{n-4}, v_{n}\}$ forms an $I_4$. 

The case when $e=v_{n}v_{n-2}$ and $e=v_iv_{n-2}$, $i \in \{2,3,\ldots, n-4\}$ are similar to the Cases $4$ and $1$ respectively.
 
Hence, for any edge $e$ of $D_n$, the graph $D_n -e$ contains one of $I_4$, $2K_2+K_1$ or $P_3+P_2$ as an induced subgraph. That is $\chi_s(D_n -e)<n-2$.
Therefore  $D_n$ is $(n-2)$-critical.
\qed
\end{proof}

\subsection{Minimum and maximum number of edges in $(n-2)$-critical graphs}

\begin{lemma}\label{lem-upper-n-2-critical}
 Let $G$ be a graph on $n \geq 5$ vertices and $m$ edges. If $\chi_s(G)=n-2$ then  $m  \geq \frac{n(n-3)}{6}$.  
\end{lemma}
\begin{proof}
Let $G$ be a graph with $\chi_s(G)=n-1$. We know from Lemma~\ref{lem-free-n-2}, $G$ is $(I_4,2K_2+K_1, P_3+P_2)$-free. Then, the complement $\overline {G}$ of $G$ is $K_4$-free. Using Tur$\acute{a}$n's~\cite{aigner1995turan} theorem, the number of edges in $\overline {G}$ is at most $\frac{n^2}{3}$. Hence, the minimum number of edges in a graph having star chromatic number $(n-2)$ is at least $ {n \choose 2} -  \frac{n^2}{3} \geq \frac{n(n-3)}{6}$. 

%
%

\end{proof}

\begin{lemma}\label{lem-not-n-2-critical}
 Let $G$ be a graph on $n \geq 5$ vertices such that $\chi_s(G)=\Delta(G)= n-2$, then $G$ is not $(n-2)$-critical.
\end{lemma}
\begin{proof}

Let $f$ be a $(n-2)$-star coloring of $G$. Let $w \in V(G)$ be a vertex having the degree $n-2$, that is $w$ is adjacent to every other vertex of $G$, except one vertex (say, $v_k$).

\paragraph{Case~1.} degree of $v_k$ is $n-2$.

Let $H=G - wv_i$, where $i \neq k$. 
We show that $H$ is $(I_4, 2K_2+K_1, P_3+P_2)$-free. If $H$ has an $I_4$ then $w \notin I_4$, as degree of $w$ in $H$ is $n-3$.  Then, the same $I_4$ is also present in $G$, which is a contradiction. If $H$ has a $2K_2+K_1$ (resp. $P_3+P_2$)  then $w \notin V(2K_2+K_1)$ (resp.  $ w \notin V(P_3+P_2)$), which implies there is a $2K_2+K_1$ (resp. $P_3+P_2$) in $G$ which is a contradiction. Therefore $H$ is $(I_4, 2K_2+K_1, P_3+P_2)$-free and from Lemma~\ref{lem-free-n-2}, we get $\chi_s(H)=n-2$.

\paragraph{Case~2.} degree of $v_k$ is at most $n-3$.

Let $v_j$ be a non-neighbor of $v_k$ in $V(G) \setminus \{w\}$.
Let $H=G - wv_j$. Similar to the above case we can easily see that $H$ is $(I_4, 2K_2+K_1, P_3+P_2)$-free. Therefore $\chi_s(H)=n-2$.

Combing the above two cases, we can conclude that $G$ is not $(n-2)$-critical. \qed
\end{proof}

\begin{lemma}\label{lem-not-n-2-critical-2nd}
 Let $G$ be a graph on $n \geq 5$ vertices such that $\chi_s(G)=n-2$ and $\Delta(G)= n-1$, then $G$ is not $(n-2)$-critical.
\end{lemma}
\begin{proof}
 The proof is similar to the proof of Lemma~\ref{lem-not-n-2-critical}.
\end{proof}


\begin{theorem}
  Let $G$ be a graph on $n \geq 5$ vertices and $m$ edges. If $G$ is $(n-2)$-critical then
 $$ n(n-3)/6 \leq  m\leq n(n-3)/2$$
\end{theorem}

\begin{proof}
 The lower bound follows from the Lemma~\ref{lem-upper-n-2-critical}. Upper bound follows from the Lemma~\ref{lem-not-n-2-critical} and Lemma~\ref{lem-not-n-2-critical-2nd}.
\end{proof}

\section{Conclusion}
In this paper, we studied about the color critical graphs with respect to star coloring. We showed that the only $3$-critical graphs are $K_3$ and $P_4$. Next, we have given a characterization of graphs whose star chromatic number is $(n-1)$ or $(n-2)$. Finally, we gave the upper and lower bounds on a minimum number of edges in $(n-1)$-critical and $(n-2)$-critical graphs.

We conclude the paper with the following list of open problems for further research.

\begin{enumerate}
 \item For a given $k$ and $n$, what is the minimum (resp. maximum) number of edges in a $k$-critical graph on $n$ vertices with respect to star coloring. 
 \item For a given $k$ and $n$, give a characterization of $(n-k)$-critical graphs with respect to star coloring. 
  
\end{enumerate}

\bibliographystyle{splncs03}
\bibliography{BibFile}

\end{document}